\documentclass[oneside,english]{amsart}
\usepackage[T1]{fontenc}
\usepackage[latin9]{inputenc}
\usepackage{prettyref}
\usepackage{enumitem}
\usepackage{amstext}
\usepackage{amsthm}
\usepackage{amssymb}
\usepackage[all]{xy}

\makeatletter
\numberwithin{equation}{section}
\numberwithin{figure}{section}
\theoremstyle{plain}
\newtheorem{thm}{\protect\theoremname}[section]
\theoremstyle{plain}
\newtheorem{lem}[thm]{\protect\lemmaname}
\newlist{casenv}{enumerate}{4}
\setlist[casenv]{leftmargin=*,align=left,widest={iiii}}
\setlist[casenv,1]{label={{\itshape\ \casename} \arabic*.},ref=\arabic*}
\setlist[casenv,2]{label={{\itshape\ \casename} \roman*.},ref=\roman*}
\setlist[casenv,3]{label={{\itshape\ \casename\ \alph*.}},ref=\alph*}
\setlist[casenv,4]{label={{\itshape\ \casename} \arabic*.},ref=\arabic*}
\theoremstyle{plain}
\newtheorem{cor}[thm]{\protect\corollaryname}

\usepackage{mathtools}
\usepackage{prettyref}
\usepackage{braket}
\usepackage{tikz}
\usetikzlibrary{patterns}

\newrefformat{prop}{Proposition \ref{#1}}
\newrefformat{cor}{Corollary \ref{#1}}
\newrefformat{exa}{Example \ref{#1}}
\newrefformat{fig}{Figure \ref{#1}}
\newrefformat{subsec}{Subsection \ref{#1}}

\newcommand{\id}{\operatorname{id}}

\makeatother

\usepackage{babel}
\providecommand{\casename}{Case}
\providecommand{\corollaryname}{Corollary}
\providecommand{\lemmaname}{Lemma}
\providecommand{\theoremname}{Theorem}

\begin{document}
\title{Another view of the coarse invariant $\sigma$}
\author{Takuma Imamura}
\address{Research Institute for Mathematical Sciences\\
Kyoto University\\
Kitashirakawa Oiwake-cho, Sakyo-ku, Kyoto 606-8502, Japan}
\email{timamura@kurims.kyoto-u.ac.jp}
\begin{abstract}
Miller, Stibich and Moore \cite{MSM10} developed a set-valued coarse
invariant $\sigma\left(X,\xi\right)$ of pointed metric spaces. DeLyser,
LaBuz and Tobash \cite{DLT13} provided a different way to construct
$\sigma\left(X,\xi\right)$ (as the set of all sequential ends). This
paper provides yet another definition of $\sigma\left(X,\xi\right)$.
To do this, we introduce a metric on the set $S\left(X,\xi\right)$
of coarse maps $\left(\mathbb{N},0\right)\to\left(X,\xi\right)$,
and prove that $\sigma\left(X,\xi\right)$ is equal to the set of
coarsely connected components of $S\left(X,\xi\right)$. As a by-product,
our reformulation trivialises some known theorems on $\sigma\left(X,\xi\right)$,
including the functoriality and the coarse invariance.
\end{abstract}

\subjclass[2000]{Primary: 40A05; Secondary: 53C23, 54D80, 54E35.}
\maketitle

\section{Introduction}

Miller, Stibich and Moore \cite{MSM10} developed a set-valued coarse
invariant $\sigma\left(X,\xi\right)$ of $\sigma$-stable pointed
metric spaces $\left(X,\xi\right)$. DeLyser, LaBuz and Wetsell \cite{DLW11}
generalised it to pointed metric spaces (without $\sigma$-stability).
The coarse invariance of $\sigma\left(X,\xi\right)$ was proved by
Fox, LaBuz and Laskowsky \cite{FLL11} for $\sigma$-stable spaces,
and by DeLyser, LaBuz and Wetsell \cite{DLW11} for general spaces.

We start with recalling the definition of $\sigma\left(X,\xi\right)$.
We adopt a simplified definition given by DeLyser, LaBuz and Tobash
\cite{DLT13}. Let $\left(X,\xi\right)$ be a pointed metric space.
A \emph{coarse sequence} in $\left(X,\xi\right)$ is a coarse map
$s\colon\left(\mathbb{N},0\right)\to\left(X,\xi\right)$. Denote the
set of coarse sequences in $\left(X,\xi\right)$ by $S\left(X,\xi\right)$.
Given $s,t\in S\left(X,\xi\right)$, we write $s\sqsubseteq_{X,\xi}^{\sigma}t$
if $s$ is a subsequence of $t$. Denote the equivalence closure of
$\sqsubseteq_{X,\xi}^{\sigma}$ by $\equiv_{X,\xi}^{\sigma}$. In
other words, $s\equiv_{X,\xi}^{\sigma}t$ if and only if there exists
a finite sequence $\set{u_{i}}_{i=0}^{n}$ in $S\left(X,\xi\right)$
such that $u_{0}=s$, $u_{n}=t$, and $u_{i}\sqsubseteq_{X,\xi}^{\sigma}u_{i+1}$
or $u_{i+1}\sqsubseteq_{X,\xi}^{\sigma}u_{i}$ for all $i<n$. The
desired invariant is then defined as the quotient set:
\begin{align*}
\sigma\left(X,\xi\right) & :=S\left(X,\xi\right)/\equiv_{X,\xi}^{\sigma}\\
 & :=\set{\left[s\right]_{X,\xi}^{\sigma}|s\in S\left(X,\xi\right)},
\end{align*}
where $\left[s\right]_{X,\xi}^{\sigma}$ is the $\equiv_{X,\xi}^{\sigma}$-equivalence
class of $s$. As noted in \cite{Ima20}, there is no difficulty in
generalising $\sigma\left(X,\xi\right)$ to pointed coarse spaces
$\left(X,\xi\right)$. See \prettyref{subsec:Notation-and-terminology}
for the definitions of the terms used here.

DeLyser, LaBuz and Tobash \cite{DLT13} provided an alternative definition
of $\sigma\left(X,\xi\right)$. Suppose $\left(X,\xi\right)$ is a
pointed metric space. Two coarse sequences $s,t\in S\left(X,\xi\right)$
are said to \emph{converge to the same sequential end} (and denoted
by $s\equiv_{X,\xi}^{e}t$) if there is a $K>0$ such that for all
bounded subsets $B$ of $X$ there is an $N\in\mathbb{N}$ such that
$\set{s\left(i\right)|i\geq N}$ and $\set{t\left(i\right)|i\geq N}$
are contained in the same $K$-chain-connected component of $X\setminus B$.
The $\equiv_{X,\xi}^{e}$-equivalence classes are called \emph{sequential
ends} in $\left(X,\xi\right)$. It was proved that $\equiv_{X,\xi}^{\sigma}$
and $\equiv_{X,\xi}^{e}$ coincides. As a result, $\sigma\left(X,\xi\right)$
is equal to the set of sequential ends in $\left(X,\xi\right)$. This
gives another view of $\sigma\left(X,\xi\right)$.

This paper aims to provide yet another view of $\sigma\left(X,\xi\right)$.
Consider the following diagram:
\[
\xymatrix{\mathsf{Coarse}_{\ast}\ar[rrr]^{\sigma}\ar[dr]_{S} &  &  & \mathsf{Sets}\\
 & \mathsf{Metr}_{b}\ar[r]_{\text{forget}} & \mathsf{Coarse}_{b}\ar[ur]_{\mathcal{Q}}
}
\]
where $\mathsf{Coarse}_{\ast}$ is the category of pointed coarse
spaces and (base point preserving) coarse maps, $\mathsf{Metr}_{b}$
the category of metric spaces and bornologous maps, $\mathsf{Coarse}_{b}$
the category of coarse spaces and bornologous maps, and $\mathsf{Sets}$
the category of sets and maps. In \prettyref{sec:Coarsely-connected-component},
we introduce the so-called coarsely connected component functor $\mathcal{Q}\colon\mathsf{Coarse}_{b}\to\mathsf{Sets}$.
The coarse invariance of $\mathcal{Q}$ is proved. In \prettyref{sec:Metrisation-of-S},
we introduce a metric on the set $S\left(X,\xi\right)$, where the
metric is allowed to take the value $\infty$. This forms a functor
$S\colon\mathsf{Coarse}_{\ast}\to\mathsf{Metr}_{b}$. We prove the
preservation of bornotopy by $S$. In \prettyref{sec:Alternative-definition-of-sigma},
we prove that $\sigma$ can be considered as the composition of the
two functors $\mathcal{Q}$ and $S$, which commutes the above diagram.
As a by-product, our reformulation trivialises some known theorems
on $\sigma\left(X,\xi\right)$, including the functoriality and the
coarse invariance.

\subsection{\label{subsec:Notation-and-terminology}Notation and terminology}

Let $f,g\colon X\to Y$ be maps, $E,F$ binary relations on $X$ (i.e.
subsets of $X\times X$), and $n\in\mathbb{N}$. Then
\begin{align*}
E\circ F & :=\set{\left(x,y\right)\in X\times X|\left(x,z\right)\in E\text{ and }\left(z,y\right)\in F\text{ for some }z\in X},\\
E^{-1} & :=\set{\left(y,x\right)\in X\times X|\left(x,y\right)\in E},\\
E^{0} & :=\Delta_{X}:=\set{\left(x,x\right)|x\in X},\\
E^{n+1} & :=E^{n}\circ E,\\
\left(f\times g\right)\left(E\right) & :=\set{\left(f\left(x\right),g\left(y\right)\right)|\left(x,y\right)\in E}.
\end{align*}

A \emph{coarse structure} on a set $X$ is a family $\mathcal{C}_{X}$
of binary relations on $X$ with the following properties:
\begin{enumerate}
\item $\Delta_{X}\in\mathcal{C}_{X}$;
\item $E\subseteq F\in\mathcal{C}_{X}\implies E\in\mathcal{C}_{X}$; and
\item $E,F\in\mathcal{C}_{X}\implies E\cup F,E\circ F,E^{-1}\in\mathcal{C}_{X}$.
\end{enumerate}
A set equipped with a coarse structure is called a \emph{coarse space}.
A subset $A$ of $X$ is called a \emph{bounded set} if $A\times A\in\mathcal{C}_{X}$.
We denote the family of bounded subsets of $X$ by $\mathcal{B}_{X}$.
This family satisfies the following:
\begin{enumerate}
\item $\bigcup\mathcal{B}_{X}=X$;
\item $A\subseteq B\in\mathcal{B}_{X}\implies A\in\mathcal{B}_{X}$;
\item $A,B\in\mathcal{B}_{X},A\cap B\neq\varnothing\implies A\cup B\in\mathcal{B}_{X}$.
\end{enumerate}
A typical example of a coarse structure is the \emph{bounded coarse
structure} induced by a metric $d_{X}\colon X\times X\to\mathbb{R}_{\geq0}\cup\set{\infty}$:
\[
\mathcal{C}_{d_{X}}:=\set{E\subseteq X\times X|\sup d_{X}\left(E\right)<\infty}\cup\set{\varnothing}.
\]
Then the boundedness defined above agrees with the usual boundedness.
Throughout this paper, we assume that every metric space is endowed
with the bounded coarse structure.

Let $f,g\colon X\to Y$ be maps from a set $X$ to a coarse space
$Y$. We say that $f$ and $g$ are \emph{bornotopic} (or \emph{close})
if $\left(f\times g\right)\left(\Delta_{X}\right)\in\mathcal{C}_{Y}$.
Obviously bornotopy gives an equivalence relation on the set $Y^{X}$
of all maps from $X$ to $Y$.

Suppose $f\colon X\to Y$ is a map between coarse spaces $X,Y$. Then
$f$ is said to be
\begin{enumerate}
\item \emph{proper} if $f^{-1}\left(B\right)\in\mathcal{B}_{X}$ for all
$B\in\mathcal{B}_{Y}$;
\item \emph{bornologous} if $\left(f\times f\right)\left(E\right)\in\mathcal{C}_{Y}$
for all $E\in\mathcal{C}_{X}$;
\item \emph{coarse} if it is proper and bornologous;
\item an \emph{asymorphism} (or an \emph{isomorphism of coarse spaces})
if it is a bornologous bijection such that the inverse map is also
bornologous;
\item a \emph{coarse equivalence} (or a \emph{bornotopy equivalence}) if
it is bornologous, and there exists a bornologous map $g\colon Y\to X$
(called a \emph{coarse inverse} or a \emph{bornotopy inverse} of $f$)
such that $g\circ f$ and $f\circ g$ are bornotopic to the identity
maps $\id_{X}$ and $\id_{Y}$, respectively.
\end{enumerate}
For more information, see the monograph \cite{Roe03} by John Roe.

\section{\label{sec:Coarsely-connected-component}Coarsely connected components}

Let $X$ be a coarse space. A subset $A$ of $X$ is said to be \emph{coarsely
connected} if $\set{x,y}\in\mathcal{B}_{X}$ for all $x,y\in A$ (\cite[Definition 2.11]{Roe03}).
For $x\in X$, we set

\[
\mathcal{Q}_{X}\left(x\right):=\bigcup_{x\in B\in\mathcal{B}_{X}}B,
\]
and call it the \emph{coarsely connected component} of $X$ containing
$x$. It is easy to see that $\mathcal{Q}_{X}\left(x\right)$ is the
largest coarsely connected subset of $X$ that contains $x$ (see
also \cite[Remark 2.20]{Roe03}). We denote the set of all coarsely
connected components of $X$ by $\mathcal{Q}\left(X\right)$:
\[
\mathcal{Q}\left(X\right):=\set{\mathcal{Q}_{X}\left(x\right)|x\in X}.
\]

\begin{lem}
Let $f\colon X\to Y$ be a bornologous map. If $X$ is coarsely connected,
then so is the image $f\left(X\right)$.
\end{lem}

\begin{proof}
Immediate from the fact that every bornologous map preserves boundedness.
\end{proof}
\begin{thm}[Functoriality]
\label{thm:functoriality-of-Q}Every bornologous map $f\colon X\to Y$
functorially induces a map $\mathcal{Q}\left(f\right)\colon\mathcal{Q}\left(X\right)\to\mathcal{Q}\left(Y\right)$
by $\mathcal{Q}\left(f\right)\left(\mathcal{Q}_{X}\left(x\right)\right):=\mathcal{Q}_{Y}\left(f\left(x\right)\right)$.
\end{thm}

\begin{proof}
It suffices to verify the well-definedness. Let $x,y\in X$ and suppose
$\mathcal{Q}_{X}\left(x\right)=\mathcal{Q}_{X}\left(y\right)$. Since
$f$ is bornologous and $\mathcal{Q}_{X}\left(x\right)$ is coarsely
connected, $f\left(\mathcal{Q}_{X}\left(x\right)\right)$ is coarsely
connected and contains $f\left(x\right)$. By the maximality of $\mathcal{Q}_{Y}\left(f\left(x\right)\right)$,
we have that $f\left(y\right)\in f\left(\mathcal{Q}_{X}\left(y\right)\right)=f\left(\mathcal{Q}_{X}\left(x\right)\right)\subseteq\mathcal{Q}_{Y}\left(f\left(x\right)\right)$.
By the maximality of $\mathcal{Q}_{Y}\left(f\left(y\right)\right)$,
we have that $\mathcal{Q}_{Y}\left(f\left(x\right)\right)\subseteq\mathcal{Q}_{Y}\left(f\left(y\right)\right)$.
By symmetry, $\mathcal{Q}_{Y}\left(f\left(y\right)\right)\subseteq\mathcal{Q}_{Y}\left(f\left(x\right)\right)$
holds. Therefore $\mathcal{Q}_{Y}\left(f\left(x\right)\right)=\mathcal{Q}_{Y}\left(f\left(y\right)\right)$.
\end{proof}
\begin{thm}[Coarse invariance]
\label{thm:coarse-invariance-of-Q}If bornologous maps $f,g\colon X\to Y$
are bornotopic, then $\mathcal{Q}\left(f\right)=\mathcal{Q}\left(g\right)$.
\end{thm}

\begin{proof}
The proof is similar to that of \prettyref{thm:functoriality-of-Q}.
Let $x\in X$. Since $f$ and $g$ are bornotopic, $\left(f\left(x\right),g\left(x\right)\right)\in\left(f\times g\right)\left(\Delta_{X}\right)\in\mathcal{C}_{Y}$,
so $\set{f\left(x\right),g\left(x\right)}$ is bounded in $Y$. Thus
$\set{f\left(x\right),g\left(x\right)}$ is coarsely connected and
contains $f\left(x\right)$. By the maximality of $\mathcal{Q}_{Y}\left(f\left(x\right)\right)$,
we have that $g\left(x\right)\in\set{f\left(x\right),g\left(x\right)}\subseteq\mathcal{Q}_{Y}\left(f\left(x\right)\right)$.
By the maximality of $\mathcal{Q}_{Y}\left(g\left(x\right)\right)$,
we have that $\mathcal{Q}_{Y}\left(f\left(x\right)\right)\subseteq\mathcal{Q}_{Y}\left(g\left(x\right)\right)$.
The reverse inclusion $\mathcal{Q}_{Y}\left(g\left(x\right)\right)\subseteq\mathcal{Q}_{Y}\left(f\left(x\right)\right)$
holds by symmetry. It follows that $\mathcal{Q}\left(f\right)\left(\mathcal{Q}_{X}\left(x\right)\right)=\mathcal{Q}_{Y}\left(f\left(x\right)\right)=\mathcal{Q}_{Y}\left(g\left(x\right)\right)=\mathcal{Q}\left(g\right)\left(\mathcal{Q}_{X}\left(x\right)\right)$.
\end{proof}

\section{\label{sec:Metrisation-of-S}Metrisation of $S\left(X,\xi\right)$}

Let $\left(X,\xi\right)$ be a pointed coarse space. A coarse map
$s\colon\left(\mathbb{N},0\right)\to\left(X,\xi\right)$ is called
a \emph{coarse sequence} in $\left(X,\xi\right)$. Denote by $S\left(X,\xi\right)$
the set of all coarse sequences of $\left(X,\xi\right)$. In the preceding
studies \cite{MSM10,FLL11,DLW11,DLT13}, $S\left(X,\xi\right)$ is
just a set with no structure. In fact, as we shall see below, $S\left(X,\xi\right)$
has a geometric structure relevant to $\sigma\left(X,\xi\right)$.
We define a metric $d_{S\left(X,\xi\right)}\colon S\left(X,\xi\right)\times S\left(X,\xi\right)\to\mathbb{N}\cup\set{\infty}$
on $S\left(X,\xi\right)$ as follows:
\[
d_{S\left(X,\xi\right)}\left(s,t\right):=\inf\Set{n\in\mathbb{N}|\left(s,t\right)\in\left(\sqsubseteq_{X,\xi}^{\sigma}\cup\sqsupseteq_{X,\xi}^{\sigma}\right)^{n}},
\]
where $\inf\varnothing:=\infty$. It is easy to check that $d_{S\left(X,\xi\right)}$
is a metric. Thus $S\left(X,\xi\right)$ is equipped with a coarse
structure, viz., the bounded coarse structure induced by $d_{S\left(X,\xi\right)}$.
\begin{lem}
\label{lem:properties-of-dS}Let $\left(X,\xi\right)$ be a pointed
coarse space and $s,t\in S\left(X,\xi\right)$.
\begin{enumerate}
\item The following are equivalent:
\begin{enumerate}
\item $s\equiv_{X,\xi}^{\sigma}t$;
\item $d_{S\left(X,\xi\right)}\left(s,t\right)\in\mathbb{N}$;
\item there exists a sequence $\set{u_{i}}_{i=0}^{n}$ in $S\left(X,\xi\right)$
of length $n+1$ such that $u_{0}=s$, $u_{n}=t$ and $d_{S\left(X,\xi\right)}\left(u_{i},u_{i+1}\right)=1$
for all $i<n$, where $n$ is an arbitrary constant greater than or
equal to $d_{S\left(X,\xi\right)}\left(s,t\right)$.
\end{enumerate}
\item The following are equivalent:
\begin{enumerate}
\item $s\not\equiv_{X,\xi}^{\sigma}t$;
\item $d_{S\left(X,\xi\right)}\left(s,t\right)=\infty$;
\item there is no finite sequence $\set{u_{i}}_{i=0}^{n}$ in $S\left(X,\xi\right)$
such that $u_{0}=s$, $u_{n}=t$ and $d_{S\left(X,\xi\right)}\left(u_{i},u_{i+1}\right)=1$
for all $i<n$.
\end{enumerate}
\end{enumerate}
\end{lem}

\begin{proof}
Notice that $d_{S\left(X,\xi\right)}\left(s,t\right)\leq n$ if and
only if there exists a sequence $\set{u_{i}}_{i=0}^{n}$ in $S\left(X,\xi\right)$
of length $n+1$ such that $u_{0}=s$, $u_{n}=t$, and $u_{i}\sqsubseteq_{X,\xi}^{\sigma}u_{i+1}$
or $u_{i+1}\sqsubseteq_{X,\xi}^{\sigma}u_{i}$ for all $i<n$; and
that $d_{S\left(X,\xi\right)}\left(s,t\right)=\infty$ if and only
if there is no such finite sequence in $S\left(X,\xi\right)$. The
above equivalences are now obvious.
\end{proof}
\begin{thm}[Functoriality]
\label{thm:functoriality-of-S}Each coarse map $f\colon\left(X,\xi\right)\to\left(Y,\eta\right)$
functorially induces a bornologous map $S\left(f\right)\colon S\left(X,\xi\right)\to S\left(Y,\eta\right)$
by $S\left(f\right)\left(s\right):=f\circ s$.
\end{thm}

\begin{proof}
Well-definedness: let $s\in S\left(X,\xi\right)$. Clearly $S\left(f\right)\left(s\right)$
is a map from $\left(\mathbb{N},0\right)$ to $\left(Y,\eta\right)$.
The class of coarse maps is closed under composition, so $S\left(f\right)\left(s\right)$
is coarse. (Let $E\in\mathcal{C}_{\mathbb{N}}$. Then $\left(s\times s\right)\left(E\right)\in\mathcal{C}_{X}$
by the bornologousness of $s$, so $\left(f\circ s\times f\circ s\right)\left(E\right)=\left(f\times f\right)\left(\left(s\times s\right)\left(E\right)\right)\in\mathcal{C}_{Y}$
by the bornologousness of $f$. Let $B\in\mathcal{B}_{Y}$. Then $f^{-1}\left(B\right)\in\mathcal{B}_{X}$
by the properness of $f$, and hence $\left(f\circ s\right)^{-1}\left(B\right)=s^{-1}\circ f^{-1}\left(B\right)\in\mathcal{B}_{\mathbb{N}}$
by the properness of $s$.) Hence $S\left(f\right)\left(s\right)\in S\left(Y,\eta\right)$.

Bornologousness: Let $s,t\in S\left(X,\xi\right)$ and suppose $d_{S\left(X,\xi\right)}\left(s,t\right)\leq n$,
i.e., there is a sequence $\set{u_{i}}_{i=0}^{n}$ in $S\left(X,\xi\right)$
of length $n+1$ such that $u_{0}=s$, $u_{n}=t$, and $u_{i}\sqsubseteq_{X,\xi}^{\sigma}u_{i+1}$
or $u_{i+1}\sqsubseteq_{X,\xi}^{\sigma}u_{i}$ for all $i<n$. Then
the sequence $\set{f\circ u_{i}}_{i=0}^{n}$ witnesses that $d_{S\left(Y,\eta\right)}\left(S\left(f\right)\left(s\right),S\left(f\right)\left(t\right)\right)=d_{S\left(Y,\eta\right)}\left(f\circ s,f\circ t\right)\leq n$.
\end{proof}
\begin{thm}[Preservation of bornotopy]
\label{thm:preservation-of-coarse-equivalence-of-S}If coarse maps
$f,g\colon\left(X,\xi\right)\to\left(Y,\eta\right)$ are bornotopic,
then so are $S\left(f\right),S\left(g\right)\colon S\left(X,\xi\right)\to S\left(Y,\eta\right)$.
\end{thm}

\begin{proof}
Let $s\in S\left(X,\xi\right)$. We define a map $t\colon\left(\mathbb{N},0\right)\to\left(Y,\eta\right)$
as follows:
\[
t\left(i\right):=\begin{cases}
S\left(f\right)\left(s\right)\left(j\right), & i=2j,\\
S\left(g\right)\left(s\right)\left(j\right), & i=2j+1.
\end{cases}
\]
Let us verify that $t\in S\left(Y,\eta\right)$. Firstly, let $B\in\mathcal{B}_{Y}$.
Then
\[
t^{-1}\left(B\right)=2\left(S\left(f\right)\left(s\right)\right)^{-1}\left(B\right)\cup\left(2\left(S\left(g\right)\left(s\right)\right)^{-1}\left(B\right)+1\right).
\]
Since $S\left(f\right)\left(s\right)$ and $S\left(g\right)\left(s\right)$
are proper, $2\left(S\left(f\right)\left(s\right)\right)^{-1}\left(B\right)$
and $2\left(S\left(g\right)\left(s\right)\right)^{-1}\left(B\right)+1$
are bounded in $\mathbb{N}$ (i.e. finite), so $t^{-1}\left(B\right)\in\mathcal{B}_{\mathbb{N}}$.
Therefore $t$ is proper. Secondly, let $n\in\mathbb{N}$. Since $S\left(f\right)\left(s\right)$
and $S\left(g\right)\left(s\right)$ are bornologous, there exists
an $E\in\mathcal{C}_{Y}$ such that $\left(S\left(f\right)\left(s\right)\left(i\right),S\left(f\right)\left(s\right)\left(j\right)\right)\in E$
and $\left(S\left(g\right)\left(s\right)\left(i\right),S\left(g\right)\left(s\right)\left(j\right)\right)\in E$
hold for all $i,j\in\mathbb{N}$ with $\left|i-j\right|\leq n$. Since
$f$ and $g$ are bornotopic,
\begin{align*}
F & :=\set{\left(S\left(f\right)\left(s\right)\left(i\right),S\left(g\right)\left(s\right)\left(i\right)\right)|i\in\mathbb{N}}\\
 & =\set{\left(f\circ s\left(i\right),g\circ s\left(i\right)\right)|i\in\mathbb{N}}\\
 & \subseteq\left(f\times g\right)\left(\Delta_{X}\right)\\
 & \in\mathcal{C}_{Y}.
\end{align*}
Then $\left(S\left(f\right)\left(s\right)\left(i\right),S\left(g\right)\left(s\right)\left(j\right)\right)\in E\circ F\in\mathcal{C}_{Y}$
and $\left(S\left(g\right)\left(s\right)\left(i\right),S\left(f\right)\left(s\right)\left(j\right)\right)\in E\circ F^{-1}\in\mathcal{C}_{Y}$
hold for all $i,j\in\mathbb{N}$ with $\left|i-j\right|\leq n$. Now
let $G:=E\cup\left(E\circ F\right)\cup\left(E\circ F^{-1}\right)\in\mathcal{C}_{Y}$.
Then $\left(t\left(i\right),t\left(j\right)\right)\in G$ holds for
all $i,j\in\mathbb{N}$ with $\left|i-j\right|\leq n$. Therefore
$t$ is bornologous.

Both $S\left(f\right)\left(s\right)$ and $S\left(g\right)\left(s\right)$
are subsequences of $t$, i.e., $S\left(f\right)\left(s\right)\sqsubseteq_{Y,\eta}^{\sigma}t\sqsupseteq_{Y,\eta}^{\sigma}S\left(g\right)\left(s\right)$,
so $d_{S\left(Y,\eta\right)}\left(S\left(f\right)\left(s\right),S\left(g\right)\left(s\right)\right)\leq2$.
Hence
\begin{align*}
\left(S\left(f\right)\times S\left(g\right)\right)\left(\Delta_{S\left(X,\xi\right)}\right) & \subseteq\set{\left(u,v\right)\in S\left(Y,\eta\right)\times S\left(Y,\eta\right)|d_{S\left(Y,\eta\right)}\left(u,v\right)\leq2}\\
 & \in\mathcal{C}_{S\left(Y,\eta\right)}.
\end{align*}
Therefore $S\left(f\right)$ and $S\left(g\right)$ are bornotopic.
\end{proof}
The next theorem shows that the base point can be replaced with any
other point lying in the same coarsely connected component.
\begin{thm}[Changing the base point]
\label{thm:changing-the-base-point}Let $X$ be a coarse space, and
$\xi_{1},\xi_{2}\in X$. If $\mathcal{Q}_{X}\left(\xi_{1}\right)=\mathcal{Q}_{X}\left(\xi_{2}\right)$,
then $S\left(X,\xi_{1}\right)$ and $S\left(X,\xi_{2}\right)$ are
isometric.
\end{thm}

\begin{proof}
Define two maps $T_{21}\colon S\left(X,\xi_{1}\right)\to S\left(X,\xi_{2}\right)$
and $T_{12}\colon S\left(X,\xi_{2}\right)\to S\left(X,\xi_{1}\right)$
by
\begin{align*}
T_{21}\left(s\right)\left(i\right) & :=\begin{cases}
\xi_{2}, & i=0,\\
s\left(i\right), & i>0,
\end{cases}\\
T_{12}\left(t\right)\left(i\right) & :=\begin{cases}
\xi_{1}, & i=0,\\
t\left(i\right), & i>0.
\end{cases}
\end{align*}
We first verify the well-definedness, i.e. $T_{21}\left(s\right)\in S\left(X,\xi_{2}\right)$
and $T_{12}\left(t\right)\in S\left(X,\xi_{1}\right)$. Obviously
$T_{21}\left(s\right)\left(0\right)=\xi_{2}$. Let $B\in\mathcal{B}_{X}$.
Then $\left(T_{21}\left(s\right)\right)^{-1}\left(B\right)\subseteq s^{-1}\left(B\right)\cup\set{0}$,
where $s^{-1}\left(B\right)$ is bounded in $\mathbb{N}$ (i.e. finite)
by the properness of $s$, so $\left(T_{21}\left(s\right)\right)^{-1}\left(B\right)$
is also bounded in $\mathbb{N}$. Hence $T_{21}\left(s\right)$ is
proper. Next, let $E\in\mathcal{C}_{\mathbb{N}}$. For each $\left(i,j\right)\in E$,
there are the following possibilities:
\begin{casenv}
\item $i=j=0$. Then $T_{21}\left(s\right)\left(i\right)=\xi_{2}=T_{21}\left(s\right)\left(j\right)$,
so $\left(T_{21}\left(s\right)\left(i\right),T_{21}\left(s\right)\left(j\right)\right)\in\Delta_{X}\in\mathcal{C}_{X}$.
\item $i=0$ and $j\neq0$. In this case, $T_{21}\left(s\right)\left(i\right)=\xi_{2}$,
$s\left(i\right)=\xi_{1}$ and $s\left(j\right)=T_{21}\left(s\right)\left(j\right)$.
So $\left(T_{21}\left(s\right)\left(i\right),T_{21}\left(s\right)\left(j\right)\right)\in\set{\left(\xi_{2},\xi_{1}\right)}\circ\left(s\times s\right)\left(E\right)\in\mathcal{C}_{X}$.
\item $i\neq0$ and $j=0$. Similarly to the above case, we have $\left(T_{21}\left(s\right)\left(i\right),T_{21}\left(s\right)\left(j\right)\right)\in\left(s\times s\right)\left(E\right)\circ\set{\left(\xi_{1},\xi_{2}\right)}\in\mathcal{C}_{X}$.
\item $i\neq0$ and $j\neq0$. Then $T_{21}\left(s\right)\left(i\right)=s\left(i\right)$
and $T_{21}\left(s\right)\left(j\right)=s\left(j\right)$, whence
we have $\left(T_{21}\left(s\right)\left(i\right),T_{21}\left(s\right)\left(j\right)\right)\in\left(s\times s\right)\left(E\right)\in\mathcal{C}_{X}$.
\end{casenv}
Set $F:=\Delta_{X}\cup\left(\set{\left(\xi_{2},\xi_{1}\right)}\circ\left(s\times s\right)\left(E\right)\right)\cup\left(\left(s\times s\right)\left(E\right)\circ\set{\left(\xi_{1},\xi_{2}\right)}\right)\cup\left(s\times s\right)\left(E\right)$.
Then $\left(T_{21}\left(s\right),T_{21}\left(s\right)\right)\left(E\right)\subseteq F\in\mathcal{C}_{X}$,
so $\left(T_{21}\left(s\right),T_{21}\left(s\right)\right)\left(E\right)\in\mathcal{C}_{X}$.
Hence $T_{21}\left(s\right)$ is bornologous. Since the definitions
are symmetric, the same argument applies to $T_{12}\left(t\right)$.

Clearly $T_{12}\circ T_{21}=\id_{S\left(X,\xi_{1}\right)}$ and $T_{21}\circ T_{12}=\id_{S\left(X,\xi_{2}\right)}$.
It suffices to prove that $T_{21}$ is an isometry.

Let $s,t\in S\left(X,\xi_{1}\right)$ and suppose that $s\sqsubseteq_{X,\xi}^{\sigma}t$,
i.e., there is a strictly monotone function $\kappa\colon\mathbb{N}\to\mathbb{N}$
such that $s=t\circ\kappa$. Since $\kappa\left(i\right)\geq i$,
we have $T_{21}\left(s\right)\left(i\right)=s\left(i\right)=t\left(\kappa\left(i\right)\right)=T_{21}\left(t\right)\left(\kappa\left(i\right)\right)$
for all $i>0$. Now, define
\[
\kappa'\left(i\right):=\begin{cases}
0, & i=0,\\
\kappa\left(i\right), & i>0.
\end{cases}
\]
Then $T_{21}\left(s\right)\left(i\right)=T_{21}\left(t\right)\left(\kappa'\left(i\right)\right)$
holds for all $i\in\mathbb{N}$ (including the case $i=0$). Hence
$T_{21}\left(s\right)\sqsubseteq_{X,\xi_{2}}^{\sigma}T_{21}\left(t\right)$.
Note that, by symmetry, the same applies to $T_{12}$. Conversely,
let $s,t\in S\left(X,\xi_{1}\right)$ and suppose $T_{21}\left(s\right)\sqsubseteq_{X,\xi_{2}}^{\sigma}T_{21}\left(t\right)$.
Then $s=T_{12}\circ T_{21}\left(s\right)\sqsubseteq_{X,\xi_{1}}^{\sigma}T_{12}\circ T_{21}\left(t\right)=t$.

Now, let $s,t\in S\left(X,\xi_{1}\right)$ and suppose $d_{S\left(X,\xi_{1}\right)}\left(s,t\right)\leq n$,
i.e., there is a sequence $\set{u_{i}}_{i=0}^{n}$ in $S\left(X,\xi\right)$
of length $n+1$ such that $u_{0}=s$, $u_{n}=t$, and $u_{i}\sqsubseteq_{X,\xi}^{\sigma}u_{i+1}$
or $u_{i+1}\sqsubseteq_{X,\xi}^{\sigma}u_{i}$ for all $i<n$. By
the previous paragraph, $T_{21}\left(u_{i}\right)\sqsubseteq_{X,\xi}^{\sigma}T_{21}\left(u_{i+1}\right)$
or $T_{21}\left(u_{i+1}\right)\sqsubseteq_{X,\xi}^{\sigma}T_{21}\left(u_{i}\right)$
for all $i<n$. Hence $d_{S\left(X,\xi_{2}\right)}\left(T_{21}\left(s\right),T_{21}\left(t\right)\right)\leq n$.
The same applies to $T_{12}$ by symmetry. Conversely, let $s,t\in S\left(X,\xi_{1}\right)$
and suppose $d_{S\left(X,\xi_{2}\right)}\left(T_{21}\left(s\right),T_{21}\left(t\right)\right)\leq n$.
Then $d_{S\left(X,\xi_{1}\right)}\left(s,t\right)=d_{S\left(X,\xi_{1}\right)}\left(T_{12}\circ T_{21}\left(s\right),T_{12}\circ T_{21}\left(t\right)\right)\leq n$.
It follows that both $T_{21}$ and $T_{12}$ are isometries.
\end{proof}
In fact, the metric function $d_{S\left(X,\xi\right)}$ only takes
the values $0$, $1$, $2$ and $\infty$. To show this fact, we need
the ``confluence'' property of $\sqsubseteq_{X,\xi}^{\sigma}$.
\begin{lem}[{DeLyser, LaBuz and Tobash \cite[Lemma 3.1]{DLT13}}]
\label{lem:One-step-confluence}Let $s,t,u\in S\left(X,\xi\right)$
and suppose $s\sqsubseteq_{X,\xi}^{\sigma}t,u$. Then there is a $v\in S\left(X,\xi\right)$
such that $t,u\sqsubseteq_{X,\xi}^{\sigma}v$.
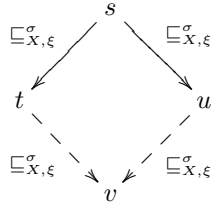
\[
\xymatrix{ & s\ar[dl]_{\sqsubseteq_{X,\xi}^{\sigma}}\ar[dr]^{\sqsubseteq_{X,\xi}^{\sigma}}\\
t\ar@{-->}[dr]_{\sqsubseteq_{X,\xi}^{\sigma}} &  & u\ar@{-->}[dl]^{\sqsubseteq_{X,\xi}^{\sigma}}\\
 & v
}
\]
\end{lem}

\begin{proof}
By the definition of ``subsequence'', there are strictly monotone
functions $\kappa,\lambda\colon\mathbb{N}\to\mathbb{N}$ such that
$s=t\circ\kappa=u\circ\lambda$. The desired sequence $v\in S\left(X,\xi\right)$
is given by
\begin{gather*}
\underset{t\left(0\right),\ldots,t\left(\kappa\left(1\right)\right)}{\underbrace{s\left(0\right),t\left(1\right),\ldots,t\left(\kappa\left(1\right)-1\right),s\left(1\right)}},\underset{u\left(0\right),\ldots,u\left(\lambda\left(1\right)\right)}{\underbrace{s\left(0\right),u\left(1\right),\ldots,u\left(\lambda\left(1\right)-1\right),s\left(1\right)}},\\
\underset{t\left(\kappa\left(1\right)\right),\ldots,t\left(\kappa\left(2\right)\right)}{\underbrace{s\left(1\right),t\left(\kappa\left(1\right)+1\right),\ldots,t\left(\kappa\left(2\right)-1\right),s\left(2\right)}},\underset{u\left(\lambda\left(1\right)\right),\ldots,u\left(\lambda\left(2\right)\right)}{\underbrace{s\left(1\right),u\left(\lambda\left(1\right)+1\right),\ldots,u\left(\lambda\left(2\right)-1\right),s\left(2\right)}},\\
\vdots
\end{gather*}
Obviously $v$ has $t$ and $u$ as subsequences. Let $E=\set{\left(i,j\right)|\left|i-j\right|\leq1}$.
(Note that $\mathcal{C}_{\mathbb{N}}$ is generated by $\set{E^{n}|n\in\mathbb{N}}$.)
Since $s$, $t$ and $u$ are bornologous, $\left(s\times s\right)\left(E\right),\left(t\times t\right)\left(E\right),\left(u\times u\right)\left(E\right)\in\mathcal{C}_{X}$.
Any two adjacent points $\left(v\left(i\right),v\left(i\pm1\right)\right)$
are one of the following forms:
\[
\left(t\left(j\right),t\left(j\pm1\right)\right),\ \left(s\left(j\right),s\left(j\pm1\right)\right),\ \left(u\left(j\right),u\left(j\pm1\right)\right),\ \left(s\left(j\right),s\left(j\right)\right),
\]
so $\left(v\times v\right)\left(E\right)\subseteq\left(t\times t\right)\left(E\right)\cup\left(s\times s\right)\left(E\right)\cup\left(u\times u\right)\left(E\right)\cup\Delta_{X}\in\mathcal{C}_{X}$.
Hence $v$ is bornologous. Similarly, one can easily prove that $v$
is proper (i.e. diverges to infinity).
\end{proof}
\begin{lem}[{DeLyser, LaBuz and Tobash \cite[Proposition 3.2]{DLT13}}]
\label{lem:Confluence}Let $s,t\in S\left(X,\xi\right)$ and suppose
$s\equiv_{X,\xi}^{\sigma}t$. Then there is a $u\in S\left(X,\xi\right)$
such that $s,t\sqsubseteq_{X,\xi}^{\sigma}u$.
\[
\xymatrix{s\ar@{-->}[dr]_{\sqsubseteq_{X,\xi}^{\sigma}}\ar@{=}[rr]^{\equiv_{X,\xi}^{\sigma}} &  & t\ar@{-->}[dl]^{\sqsubseteq_{X,\xi}^{\sigma}}\\
 & u
}
\]
\end{lem}

\begin{proof}
Choose a sequence $\set{u_{i}}_{i=0}^{n}$ in $S\left(X,\xi\right)$
such that $u_{0}=s$, $u_{n}=t$, and $u_{i}\sqsubseteq_{X,\xi}^{\sigma}u_{i+1}$
or $u_{i+1}\sqsubseteq_{X,\xi}^{\sigma}u_{i}$ for all $i<n$. We
show that there is a $v\in S\left(X,\xi\right)$ such that $u_{0},u_{n}\sqsubseteq_{X,\xi}^{\sigma}v$
by induction on the length $n$. The base case $n=0$ is trivial.
Suppose $n>0$. Since $u_{0}\equiv_{X,\xi}^{\sigma}u_{n-1}$, there
is a $v\in S\left(X,\xi\right)$ such that $u_{0},u_{n-1}\sqsubseteq_{X,\xi}^{\sigma}v$
by the induction hypothesis.
\begin{casenv}
\item $u_{n-1}\sqsubseteq_{X,\xi}^{\sigma}u_{n}$. Since $u_{n-1}\sqsubseteq_{X,\xi}^{\sigma}u_{n},v$,
there is a $v'\in S\left(X,\xi\right)$ such that $u_{n},v\sqsubseteq_{X,\xi}^{\sigma}v'$
by \prettyref{lem:One-step-confluence}. Then $u_{0}\sqsubseteq_{X,\xi}^{\sigma}v\sqsubseteq_{X,\xi}^{\sigma}v'$,
so $u_{0}\sqsubseteq_{X,\xi}^{\sigma}v'$.
\[
\xymatrix{ & u_{0}\ar[dr]\ar@{=}[rr] &  & u_{n-1}\ar[dl]\ar[dr]\\
 &  & v\ar[dr] &  & u_{n}\ar[dl]\\
 &  &  & v'
}
\]
\item $u_{n-1}\sqsupseteq_{X,\xi}^{\sigma}u_{n}$. Then $u_{n}\sqsubseteq_{X,\xi}^{\sigma}u_{n-1}\sqsubseteq_{X,\xi}^{\sigma}v$,
so $u_{n}\sqsubseteq_{X,\xi}^{\sigma}v$.
\[
\xymatrix{ &  &  & u_{n}\ar[dl]\\
u_{0}\ar@{=}[rr]\ar[dr] &  & u_{n-1}\ar[dl]\\
 & v
}
\]
\end{casenv}
\end{proof}
\begin{thm}
$d_{S\left(X,\xi\right)}\colon X\times X\to\set{0,1,2,\infty}$.
\end{thm}

\begin{proof}
Let $s,t\in S\left(X,\xi\right)$ and suppose $s\equiv_{X,\xi}^{\sigma}t$.
There is a $u\in S\left(X,\xi\right)$ such that $s,t\sqsubseteq_{X,\xi}^{\sigma}u$
by \prettyref{lem:Confluence}. Hence $d_{S\left(X,\xi\right)}\left(s,t\right)\le2$.
\end{proof}
A similar argument in \prettyref{lem:Confluence} is often used in
the context of rewriting systems (such as lambda calculus). See also
\cite[Chapter 6]{BN98}.

\section{\label{sec:Alternative-definition-of-sigma}Alternative definition
of $\sigma$}

Our main theorem is the following. This gives an alternative definition
of $\sigma\left(X,\xi\right)$ in terms of the coarse structure of
$S\left(X,\xi\right)$.
\begin{thm}
\label{thm:sigma-is-Q-S}Let $\left(X,\xi\right)$ be a pointed coarse
space. Then $\left[s\right]_{X,\xi}^{\sigma}=\mathcal{Q}_{S\left(X,\xi\right)}\left(s\right)$
for all $s\in S\left(X,\xi\right)$. Hence $\sigma\left(X,\xi\right)=\mathcal{Q}\left(S\left(X,\xi\right)\right)$.
\end{thm}

\begin{proof}
Let $s\in S\left(X,\xi\right)$. Then, by \prettyref{lem:properties-of-dS}-(1),
$\left[s\right]_{X,\xi}^{\sigma}$ is coarsely connected (in fact,
$1$-chain-connected) as a subset of $S\left(X,\xi\right)$, and contains
$s$. Hence $\left[s\right]_{X,\xi}^{\sigma}\subseteq\mathcal{Q}_{S\left(X,\xi\right)}\left(s\right)$
by the maximality of $\mathcal{Q}_{S\left(X,\xi\right)}\left(s\right)$.
Conversely, let $t\in\mathcal{Q}_{S\left(X,\xi\right)}\left(s\right)$.
By \prettyref{lem:properties-of-dS}-(2), $s\equiv_{X,\xi}^{\sigma}t$
must hold, and therefore $t\in\left[s\right]_{X,\xi}^{\sigma}$. Hence
$\mathcal{Q}_{S\left(X,\xi\right)}\left(s\right)\subseteq\left[s\right]_{X,\xi}^{\sigma}$.
\end{proof}
This yields quite simple and systematic proofs of some existing results
on $\sigma\left(X,\xi\right)$.
\begin{thm}
\label{thm:functoriality-of-sigma}Each coarse map $f\colon\left(X,\xi\right)\to\left(Y,\eta\right)$
functorially induces a map $\sigma\left(f\right)\colon\sigma\left(X,\xi\right)\to\sigma\left(Y,\eta\right)$
by $\sigma\left(f\right)\left(\left[s\right]_{X,\xi}^{\sigma}\right):=\left[f\circ s\right]_{Y,\eta}^{\sigma}$.
\end{thm}

\begin{proof}
Immediate from \prettyref{thm:functoriality-of-Q}, \prettyref{thm:functoriality-of-S}
and \prettyref{thm:sigma-is-Q-S}.
\end{proof}
\begin{cor}[{Miller, Stibich and Moore \cite[Theorem 10]{MSM10}}]
If pointed coarse spaces $\left(X,\xi\right)$ and $\left(Y,\eta\right)$
are asymorphic, then $\sigma\left(X,\xi\right)\cong\sigma\left(Y,\eta\right)$.
\end{cor}

\begin{proof}
Obvious from the fact that every functor preserves isomorphisms.
\end{proof}
\begin{thm}
\label{thm:bornotopy-invariance-of-sigma}If coarse maps $f,g\colon\left(X,\xi\right)\to\left(Y,\eta\right)$
are bornotopic, then $\sigma\left(f\right)=\sigma\left(g\right)$.
\end{thm}

\begin{proof}
Immediate from \prettyref{thm:preservation-of-coarse-equivalence-of-S},
\prettyref{thm:coarse-invariance-of-Q} and \prettyref{thm:sigma-is-Q-S}.
\end{proof}
\begin{cor}[{DeLyser, LaBuz and Wetsell \cite[Theorem 4]{DLW11}}]
If pointed coarse spaces $\left(X,\xi\right)$ and $\left(Y,\eta\right)$
are coarsely equivalent, then $\sigma\left(X,\xi\right)\cong\sigma\left(Y,\eta\right)$.
\end{cor}

\begin{proof}
Let $f\colon\left(X,\xi\right)\to\left(Y,\eta\right)$ be a coarse
equivalence with a coarse inverse $g\colon\left(Y,\eta\right)\to\left(X,\xi\right)$.
Then $f\circ g$ and $g\circ f$ are bornotopic to $\id_{\left(Y,\eta\right)}$
and $\id_{\left(X,\xi\right)}$, respectively. By \prettyref{thm:functoriality-of-sigma}
and \prettyref{thm:bornotopy-invariance-of-sigma},
\begin{align*}
\id_{\sigma\left(Y,\eta\right)} & =\sigma\left(\id_{\left(Y,\eta\right)}\right)\\
 & =\sigma\left(f\circ g\right)\\
 & =\sigma\left(f\right)\circ\sigma\left(g\right),\\
\id_{\sigma\left(X,\xi\right)} & =\sigma\left(\id_{\left(X,\xi\right)}\right)\\
 & =\sigma\left(g\circ f\right)\\
 & =\sigma\left(g\right)\circ\sigma\left(f\right),
\end{align*}
so $\sigma\left(f\right)$ and $\sigma\left(g\right)$ are inverse
to each other. Hence $\sigma\left(X,\xi\right)\cong\sigma\left(Y,\eta\right)$.
\end{proof}
\begin{cor}[{DeLyser, LaBuz and Wetsell \cite[Proposition 3]{DLW11}}]
Let $X$ be a coarse space, and $\xi_{1},\xi_{2}\in X$. If $\mathcal{Q}_{X}\left(\xi_{1}\right)=\mathcal{Q}_{X}\left(\xi_{2}\right)$,
then $\sigma\left(X,\xi_{1}\right)$ and $\sigma\left(X,\xi_{2}\right)$
are equipotent (i.e. have the same cardinality).
\end{cor}

\begin{proof}
By \prettyref{thm:changing-the-base-point}, $S\left(X,\xi_{1}\right)$
and $S\left(X,\xi_{2}\right)$ are isometric and hence asymorphic.
So $\sigma\left(X,\xi_{1}\right)=\mathcal{Q}\left(S\left(X,\xi_{1}\right)\right)\cong\mathcal{Q}\left(S\left(X,\xi_{2}\right)\right)=\sigma\left(X,\xi_{2}\right)$
by \prettyref{thm:functoriality-of-Q} and \prettyref{thm:sigma-is-Q-S}.
\end{proof}

\subsection*{Acknowledgement}

The author is grateful to the anonymous referee for valuable comments
which improved the quality of the manuscript. The referee pointed
out that the metric function $d_{S\left(X,\xi\right)}$ only takes
the values $0$, $1$, $2$ and $\infty$ by \cite[Proposition 3.2]{DLT13}.

\bibliographystyle{amsplain}
\bibliography{Another_view}

\end{document}